\documentclass[10pt,reqno]{amsart}
\usepackage[top=2cm,bottom=2cm,right=2.5cm,left=2.5cm]{geometry}
\usepackage{amssymb}
\usepackage{amsmath, amsthm, amscd, amsfonts, amssymb, graphicx, color}
\usepackage[bookmarksnumbered, colorlinks, plainpages]{hyperref}
\hypersetup{colorlinks=true,linkcolor=red, anchorcolor=green, citecolor=cyan, urlcolor=red, filecolor=magenta, pdftoolbar=true}
 
\usepackage{hyperref}

\newtheorem{theorem}{Theorem}[section]
\newtheorem{lemma}[theorem]{Lemma}
\newtheorem{proposition}[theorem]{Proposition}
\newtheorem{corollary}[theorem]{Corollary}
\theoremstyle{definition}
\newtheorem{definition}[theorem]{Definition}
\newtheorem{example}[theorem]{Example}

\theoremstyle{remark}

\numberwithin{equation}{section}

\begin{document}
	
	\setcounter{page}{1}
	
	\title[Sum of $g-$frames in Hilbert $C^{\ast}-$modules ]{Sum of $g-$frames in Hilbert $C^{\ast}-$modules}

	\author[A. Lfounoune, H. Massit, A. Karara, M. Rossafi]{Abdellatif Lfounoune$^{1*}$, Hafida Massit$^{1}$, Abdelilah Karara$^{1}$ and Mohamed Rossafi$^{2}$}
	
	\address{$^{1}$Laboratory Partial Differential Equations, Spectral Algebra and Geometry, University of Ibn Tofail, Kenitra, Morocco}
	\email{\textcolor[rgb]{0.00,0.00,0.84}{abdellatif.lfounoune@uit.ac.ma; massithafida@yahoo.fr; abdelilah.karara.sm@gmail.com}}
	
	\address{$^{2}$Laboratory Partial Differential Equations, Spectral Algebra and Geometry, Higher School of Education and Training, University of Ibn Tofail, Kenitra, Morocco}
	\email{\textcolor[rgb]{0.00,0.00,0.84}{mohamed.rossafi@uit.ac.ma}}
	
	
	
	\date{%10/03/2020; %Accepted: zzzzzz.
		\newline \indent $^{*}$Corresponding author}

	\begin{abstract} In this article, we study g-frames in Hilbert $C^*$-modules and investigate conditions under which the sum of two g-frames (or a g-frame and a g-Bessel sequence) remains a g-frame. We also address the stability of g-frames under certain perturbations and provide illustrative examples in the context of $C^*$-algebras. Our results unify and extend many of the existing theorems on g-frames, focusing on the invertibility of associated operators as a key condition for guaranteeing that sums of g-frames preserve the g-frame property.
	\end{abstract}
	\maketitle
	\section{Introduction }
	
Frames for classical Hilbert spaces were introduced by Duffin and Schaeffer \cite{Duf} in 1952 to research certain difficult nonharmonic Fourier series problems, following the fundamental paper \cite{DGM} by Daubechies, Grossman and Meyer, frame theory started to become popular, especially in the more specific context of Gabor frames and wavelet frames. 
Wenchang Sun \cite{S} introduced the generalized frame, or $g-$frame, for a Hilbert space. Recently, D. Li and J. Leng \cite{Li} introduced Operator representations of $g-$frames in Hilbert spaces. Controlled $g-$frames in Hilbert C*-modules have been investigated by N. K. Sahu \cite{Sah}. The notion of Approximate Oblique Dual $g-$frames for Closed Subspaces of Hilbert Spaces can be found in \cite{Chi}. Inspired by aspects of frames and $g-$frames, We give new results on $g-$frames for Hilbert $ C^{\ast} -$modules. For more detailed information on frame theory, readers are recommended to consult: \cite{Assila, Ghiati, Karara, Massit, FR1, RFDCA, Kabbaj,  Chouchene, NhariRossTou, r1, r3, r5, r6}.

Hilbert $C^{\ast}$-modules are generalizations of Hilbert spaces by allowing the inner product to take values in a $C^{\ast}$-algebra rather than in the field of complex
numbers.

Let's now review the definition of a Hilbert $C^{\ast}$-module, the basic properties and some facts concerning operators on Hilbert $C^{\ast}$-module.

	\begin{definition} \cite{Kal}
		Let $\mathcal{A}$ be a unital $ C^{\ast}- $ algebra and $ \mathcal{H} $ be a left $ \mathcal{A}-$ module, such that the linear structures of $ \mathcal{A} $ and $ \mathcal{H} $ are compatible. $ \mathcal{H} $ is a pre-Hilbert $ \mathcal{A} $ module if $ \mathcal{H} $ is equipped with an $ \mathcal{A}- $valued inner product $ \langle \cdot, \cdot\rangle_{\mathcal{A}}  : \mathcal{H} \times \mathcal{H}\rightarrow \mathcal{A}$ such that is sesquilinear, positive definite and respects the module action. In the other words,
		\begin{itemize}
			\item [1]- $\langle x, x\rangle_{\mathcal{A}}  \geq 0$,    $ \forall x\in \mathcal{H}$ and $\langle x, x\rangle_{\mathcal{A}}  = 0$ if and only if $ x=0 $.
			\item [2]- $\langle ax+y, z\rangle_{\mathcal{A}}  =a\langle x,z\rangle_{\mathcal{A}} +\langle y,z\rangle_{\mathcal{A}}  $ for all $ a\in \mathcal{A} $ and $ x,y,z \in \mathcal{H} $
		\item[3]- $ \langle x,y \rangle_{\mathcal{A}}  =  \langle y,x \rangle_{\mathcal{A}} ^{\ast}$ for all $ x,y \in\mathcal{H} $.
		\end{itemize}
	For $ x\in\mathcal{H} $, we define $\Vert x\Vert = \Vert \langle x,x \rangle_{\mathcal{A}}  \Vert^{\frac{1}{2}}$. If $ \mathcal{H} $ is complete with $ \Vert .\Vert $, it is called a Hilbert $ \mathcal{A}- $module or a Hilbert $ C^{\ast} -$module over $ \mathcal{A} $. For every $a$ in $ C^{\ast}- $algebra $ \mathcal{A} $, we have  $ \vert a \vert = (a^{\ast}a)^{\frac{1}{2}} $ and the $ \mathcal{A}- $valued norm on $ \mathcal{H} $ is defined by $ \vert x\vert= \langle x,x\rangle_{\mathcal{A}} ^{\frac{1}{2}} $ for $ x\in\mathcal{H} $.
	\end{definition}
	\begin{lemma} \cite{Pas}
	Let $\mathcal{H}$ be a Hilbert $\mathcal{A}$-module. If $\mathcal{T}\in End_{\mathcal{A}}^{\ast}(\mathcal{H})$, then $$\langle \mathcal{T}x,\mathcal{T}x\rangle_{\mathcal{A}}\leq\|\mathcal{T}\|^{2}\langle x,x\rangle_{\mathcal{A}}, \forall x\in\mathcal{H}.$$
\end{lemma}
\begin{lemma}\cite{Ara}\label{L22}
Let $\mathcal{H}$ and $\mathcal{K}$ two Hilbert $\mathcal{A}$-modules and $\mathcal{T} \in \operatorname{End}^*(\mathcal{H}, \mathcal{K})$. Then the following statements are equivalent
\begin{enumerate}
\item $\mathcal{T}$ is surjective.
\item  $\mathcal{T}^*$ is bounded below with respect to norm, i.e., there is $m>0$ such that $m\|x\| \leq\left\|\mathcal{T}^* x\right\|$ for all $x \in \mathcal{K}$.
\item $\mathcal{T}^*$ is bounded below with respect to the inner product, i.e., there is $m^{\prime}>0$ such that $m^{\prime}\langle x, x\rangle \leq\left\langle \mathcal{T}^* x, \mathcal{T}^* f\right\rangle$ for all $x \in\mathcal{K}$.
\end{enumerate}
\end{lemma}

Throughout the paper, let $\Theta$ be a finite or countably index set and we consider $ \mathcal{H} $, $ \mathcal{K} $ be two  Hilbert $ \mathcal{A} -$modules. A  map $  \mathcal{T}:  \mathcal{H}   \rightarrow   \mathcal{K} $ is said to be adjointable if there exists a map $  \mathcal{T}^{\ast}:  \mathcal{K}   \rightarrow   \mathcal{H} $ such that $ \langle \mathcal{T} x, y\rangle_{\mathcal{A}} = \langle  x,\mathcal{T}^{\ast} y\rangle_{\mathcal{A}}$ for all $ x\in \mathcal{H} $ and $ y\in \mathcal{K} $. Let $\mathcal{H}_\zeta$ be a family of Hilbert $\mathcal{A}$-modules indexed by
$\zeta\in \Theta $, the collection of all adjointable $\mathcal{A}$-linear maps from $\mathcal{H}$ to $\mathcal{H}_{\zeta}$ is denoted by $\operatorname{End}_\mathcal{A}^*\left(\mathcal{H}, \mathcal{H}_\zeta\right)$ and $\operatorname{End}_\mathcal{A}^*(\mathcal{H}, \mathcal{H})$ is abbreviated to $\operatorname{End}_\mathcal{A}^*(\mathcal{H})$.

\begin{definition}\cite{Kal} A collection $ \{x_{\zeta}\}_{\zeta\in \Theta}\subset \mathcal{H} $ is called a frame for $ \mathcal{H} $. If there exist two positive constants $ A $ and $ B $ such that for all $ x\in \mathcal{H} $, 
	\begin{equation}\label{eq1}
		A\langle x,x\rangle_{\mathcal{A}}\leq \sum_{\zeta\in \Theta}  \langle x,x_{\zeta}\rangle_{\mathcal{A}} \langle x_{\zeta},x\rangle_{\mathcal{A}} \leq B \langle f,f\rangle_{\mathcal{A}},
	\end{equation} 
\end{definition}

\begin{definition}\cite{Kh} A sequence $ \{\Psi_{\zeta} \in \operatorname{End}_{\mathcal{A}}^{\ast}(\mathcal{H},\mathcal{H}_{\zeta}): \zeta\in \Theta\} $ is said a $g-$frame for $ \mathcal{H} $ with respect to $ \{\mathcal{H}_{\zeta}: \zeta\in \Theta\} $. If there exist two positive constants $ A $ and $ B $ such that for $ x\in\mathcal{H} $,
	\begin{equation} \label{eq2}
			A\langle x,x\rangle_{\mathcal{A}}\leq \sum_{\zeta\in \Theta}  \langle \Psi_{\zeta}x,\Psi_{\zeta}x\rangle_{\mathcal{A}}  \leq B \langle x,x\rangle_{\mathcal{A}}.
	\end{equation}
The numbers $ A $ and $ B $ are called lower and upper  bounds of the $g-$frame, respectively. If $ A = B = \nu $, the $g-$frame is $ \nu- $tight. If $ A= B=1 $, it is called a $g$-Parseval frame. If olny the right-hand inequality (\ref{eq2}) is satisfied, the family  $ \{\Psi_{\zeta} \in \operatorname{End}_{\mathcal{A}}^{\ast}(\mathcal{H},\mathcal{H}_{\zeta}): \zeta\in \Theta\} $ is called a $g-$Bessel sequence for $\mathcal{H}$ with $g-$Bessel bound $B$.
\end{definition}

Let $ \{\Psi_{\zeta} \in \operatorname{End}_{\mathcal{A}}^{\ast}(\mathcal{H},\mathcal{H}_{\zeta}): \zeta\in \Theta\} $ be a $ g- $Bessel sequence for $ \mathcal{H} $.
We define the analysis  $ \mathcal{T}^{\ast} $, the synthesis operator $\mathcal{T}$  and the $ g- $frame operator $ S $ as follows: $ \mathcal{T}^{\ast}: \mathcal{H}\rightarrow \bigoplus_{\zeta\in \Theta}\mathcal{H}_{\zeta} $, $ \mathcal{T}^{\ast}x=\{ \Psi_{\zeta}x\}_{\zeta\in \Theta} $, $ \mathcal{T}: \bigoplus_{\zeta\in \Theta}\mathcal{H}_{\zeta} \rightarrow \mathcal{H} $, $ \mathcal{T} \left(\{y_{\zeta}\}_{\zeta\in \Theta}\right)= \sum\limits_{\zeta\in \Theta}\Psi_{\zeta}^{\ast}y_{\zeta}$, and $ S:\mathcal{H} \rightarrow \mathcal{H} $ is given by $ Sx= \sum\limits_{\zeta\in \Theta}\Psi_{\zeta}^{\ast}\Psi_{\zeta}x $ for all $ x\in \mathcal{H} $.

\section{Main result}
 
  \begin{proposition}\label{p1} Let $ \{\Psi_{\zeta} \in \operatorname{End}_{\mathcal{A}}^{\ast}(\mathcal{H},\mathcal{H}_{\zeta}): \zeta\in \Theta\} $ be a $ g- $Bessel sequence for $ \mathcal{H} $ with the $ g- $frame operator $ S $. Then, $ \{\Psi_{\zeta}\}_{\zeta\in \Theta} $ is a $ g -$frame for $ \mathcal{H} $ with respect to $ \{\mathcal{H}_{\zeta}\}_{\zeta\in \Theta} $ if and only if $ S\geq \alpha I $ for some $ \alpha >0 $.
  	
  	\end{proposition}
  \begin{proof}Suppose $ \{\Psi_{\zeta}\}_{\zeta\in \Theta} $ is a $ g- $frame for  $ \mathcal{H} $, we have for some $ \alpha >0 $,
  	\begin{equation*}
  		  \alpha\langle x,x\rangle_{\mathcal{A}}\leq \sum_{\zeta\in \Theta}  \langle \Psi_{\zeta}x,\Psi_{\zeta}x\rangle_{\mathcal{A}},\: \forall x \in \mathcal{H}.  
  	\end{equation*} 
For any $x\in\mathcal{H}$, we have
	\begin{equation*}
\sum_{\zeta\in \Theta}  \langle \Psi_{\zeta}x,\Psi_{\zeta}x\rangle_{\mathcal{A}}=\sum_{\zeta\in \Theta}  \langle \Psi_{\zeta}^{\ast}\Psi_{\zeta}x,x\rangle_{\mathcal{A}}=\langle Sx,x\rangle_{\mathcal{A}}.  
\end{equation*} 

Therefore, $ \alpha I\leq S $.

  	Suppose $ \{\Psi_{\zeta}\}_{\zeta\in \Theta} $ is a $ g- $Bessel sequence for $ \mathcal{H} $ and there exists some $ \alpha>0 $ such that $ S\geq \alpha I $. Thus, we have
  	
  	$$ \langle  \alpha I x,x\rangle_{\mathcal{A}}\leq  \langle  S x,x\rangle_{\mathcal{A}}$$  for all $ x\in  \mathcal{H},$ this implies $$ \langle  \alpha I x,x\rangle_{\mathcal{A}}\leq \langle  \sum_{\zeta\in \Theta}\Psi_{\zeta}^{\ast}\Psi_{\zeta} x,x\rangle_{\mathcal{A}}.$$ 
  	
  	Then, the lower condition holds:
  	\begin{equation*}
  		\alpha  \langle x,x\rangle_{\mathcal{A}} \leq  \sum_{\zeta\in \Theta}  \langle \Psi_{\zeta}x,\Psi_{\zeta}x\rangle_{\mathcal{A}},\; \forall x\in\mathcal{H}.
  	\end{equation*}
  	\end{proof}
  	\begin{corollary}\label{c32}
  	The $g-$Bessel sequence $ \{\Psi_{\zeta} \in \operatorname{End}_{\mathcal{A}}^{\ast}(\mathcal{H},\mathcal{H}_{\zeta}): \zeta\in \Theta\} $ with the $g-$frame operator $S$ is a $g-$frame for $\mathcal{H}$ if and only if $S>0$.
  	\end{corollary}
  	\begin{lemma}\label{L33}
  	Let $ \{\Psi_{\zeta} \in \operatorname{End}_{\mathcal{A}}^{\ast}(\mathcal{H},\mathcal{H}_{\zeta}): \zeta\in \Theta\} $ be a $g-$Bessel sequence for $\mathcal{H}$, with the
synthesis operator $\mathcal{T}$ and the $g-$frame operator $ S $. Then  $\mathcal{T}$ is surjective if only if $ S>0 $.
  	\end{lemma}
\begin{proof}
 Let $ \{\Psi_{\zeta} \in \operatorname{End}_{\mathcal{A}}^{\ast}(\mathcal{H},\mathcal{H}_{\zeta}): \zeta\in \Theta\} $ be a $g-$Bessel sequence for $\mathcal{H}$. Assume that $\mathcal{T}$ is surjective, by Lemma \ref{L22}  there is $\alpha>0$ such that for any $x\in \mathcal{H}$,
 $$ \left\langle \mathcal{T}^* x, \mathcal{T}^* x\right\rangle_{\mathcal{A}}\geq\alpha\langle x, x\rangle_{\mathcal{A}}.$$
Therefore 
$$\left\langle \mathcal{T} \mathcal{T}^* f, f\right\rangle_{\mathcal{A}}\geq\alpha\langle f, f\rangle_{\mathcal{A}}.$$
Hence $S>\alpha I$. So by Proposition \ref{p1} and Corollary \ref{c32} we conclude that $S>0$.

Conversely suppose that the $g-$frame operator  $ S>0 $, by Corollary \ref{c32} $ \{\Psi_{\zeta} \in \operatorname{End}_{\mathcal{A}}^{\ast}(\mathcal{H},\mathcal{H}_{\zeta}): \zeta\in \Theta\} $ is a $g-$frame for $\mathcal{H}$. According to Theorem 3.2 in \cite{Xia}, we conclude that $\mathcal{T}$ is surjective.
\end{proof}
\begin{theorem}Let $ \{\Psi_{\zeta} \in \operatorname{End}_{\mathcal{A}}^{\ast}(\mathcal{H},\mathcal{H}_{\zeta}): \zeta\in \Theta\} $ be a $ g- $frame for $ \mathcal{H} $ with bounds $ A,\; B $, with the $ g- $frame operator $ S $. If $\Lambda \in \operatorname{End}_{\mathcal{A}}^{\ast}(\mathcal{H}) $ such that $( I+\Lambda)^{\ast}S(I+\Lambda)\geq S$. Then $ \{\Psi_{\zeta}+\Psi_{\zeta}\Lambda\}_{\zeta\in \Theta} $ is a $ g- $frame for $ \mathcal{H} $. 
	
\end{theorem}
\begin{proof}Assume that $ \{\Psi_{\zeta}\}_{\zeta\in \Theta} $ be a $ g- $frame and  $\Lambda \in \operatorname{End}_{\mathcal{A}}^{\ast}(\mathcal{H}) $ is bounded. 
	We have for any $ x\in\mathcal{H} $, 
	\begin{equation*}
			A\langle  x,x\rangle_{\mathcal{A}} \leq  \sum_{\zeta\in \Theta}  \langle \Psi_{\zeta}x,\Psi_{\zeta}x\rangle_{\mathcal{A}}\leq B\langle  x,x\rangle_{\mathcal{A}}
	\end{equation*}
	and 
	\begin{align*}
		   \langle( \Psi_{\zeta}+\Lambda \Psi_{\zeta})x,(\Psi_{\zeta}+\Lambda \Psi_{\zeta})x\rangle_{\mathcal{A}}
&=  \langle  \Psi_{\zeta}x,\Psi_{\zeta} x\rangle_{\mathcal{A}}+   \langle  \Psi_{\zeta}\Lambda x,\Psi_{\zeta}\Lambda x\rangle_{\mathcal{A}} +   \langle  \Psi_{\zeta}x,\Psi_{\zeta}\Lambda x\rangle_{\mathcal{A}} \\&+   \langle  \Psi_{\zeta}\Lambda x,\Psi_{\zeta} x\rangle_{\mathcal{A}}\\
&\leq 2 \left(\langle  \Psi_{\zeta}x,\Psi_{\zeta} x\rangle_{\mathcal{A}}+   \langle  \Psi_{\zeta}\Lambda x,\Psi_{\zeta}\Lambda x\rangle_{\mathcal{A}} \right)
	\end{align*}
Therefore	
	\begin{align*}
		  \sum_{\zeta\in \Theta}  \langle( \Psi_{\zeta}+\Lambda \Psi_{\zeta})x,(\Psi_{\zeta}+\Lambda \Psi_{\zeta})x\rangle_{\mathcal{A}}
&\leq 2\left( \sum_{\zeta\in \Theta}  \langle  \Psi_{\zeta}x,\Psi_{\zeta} x\rangle_{\mathcal{A}}+ \sum_{\zeta\in \Theta}  \langle  \Psi_{\zeta}\Lambda x,\Psi_{\zeta}\Lambda x\rangle_{\mathcal{A}}\right)\\
&\leq 2B(1+\Vert \Lambda \Vert^{2})\langle  x,x\rangle_{\mathcal{A}}.
	\end{align*}
Thus,  $ \{\Psi_{\zeta}+\Psi_{\zeta}\Lambda\}_{\zeta\in \Theta} $ is a $ g- $Bessel sequence. Since $ \{\Psi_{\zeta}\}_{\zeta\in \Theta} $ is a $ g- $frame for  $ \mathcal{H} $, by Proposition \ref{p1}, for some $ \alpha >0$, we have $ S\geq \alpha I $. Also, for every $ x\in \mathcal{H} $, we get 
\begin{align*}
	 \sum_{\zeta\in \Theta} (  \Psi_{\zeta}+\Lambda \Psi_{\zeta})^{\ast}(\Psi_{\zeta}+\Lambda \Psi_{\zeta})&=  \sum_{\zeta\in \Theta}  \left( \Psi_{\zeta}^{\ast}\Psi_{\zeta}+ \Psi_{\zeta}^{\ast}\Psi_{\zeta}\Lambda \right. +\left.  \Lambda^{\ast}\Psi_{\zeta}^{\ast}\Psi_{\zeta}+ \Lambda^{\ast}\Psi_{\zeta}^{\ast}\Psi_{\zeta}\Lambda \right) \\
	 & =  S+S\Lambda+\Lambda^{\ast}S+\Lambda^{\ast}S\Lambda\\
	 &=( I+\Lambda)^{\ast}S(I+\Lambda)
\end{align*} 
Then, 
$$ \mathcal{P}= S+S\Lambda+\Lambda^{\ast}S+\Lambda^{\ast}S\Lambda=( I+\Lambda)^{\ast}S(I+\Lambda)\geq S $$
 Hence the operator $\mathcal{P}$  is the $ g- $frame operator of $ \{\Psi_{\zeta}+\Psi_{\zeta}\Lambda\}_{\zeta\in \Theta} $. So, $ \mathcal{P}\geq \alpha I $ for some $ \alpha>0 $. Using proposition \ref{p1} we obtain $ \{\Psi_{\zeta}+\Psi_{\zeta}\Lambda\}_{\zeta\in \Theta} $ is a $ g- $frame for $ \mathcal{H} $ with respect to $ \{\mathcal{H}_{\zeta}\}_{\zeta\in \Theta} $.
\end{proof}

\begin{theorem} \label{t3}
	Let $\Psi= \{\Psi_{\zeta}\}_{\zeta\in \Theta} $ and $ \Delta= \{\Delta_{\zeta}\}_{\zeta\in \Theta} $  are two $ g- $Bessel sequences for $ \mathcal{H} $ with the synthesis operators $ \mathcal{T}_{\Psi} $, $ \mathcal{T}_{\Delta} $ and the $ g- $frame operators $ S_{\Psi} $, $ S_{\Delta} $, respectively. If $ M $, $ N \in \operatorname{End}_{\mathcal{A}}^{\ast}(\mathcal{H}) $, then the following conditions are equivalents: 
	\begin{itemize}
		\item [(1)]  $ \{\Psi_{\zeta}M+\Delta_{\zeta}N\}_{\zeta\in \Theta} $ is a $ g- $frame for $ \mathcal{H} $.
		\item[(2)] $ M^{\ast}\mathcal{T}_{\Psi}+N\mathcal{T}_{\Delta} $ is surjective.
		\item[(3)] $ S= M^{\ast}S_{\Psi}M+M^{\ast}\mathcal{T}_{\Psi}\mathcal{T}^{\ast}_{\Delta}N+N^{\ast}\mathcal{T}_{\Delta}\mathcal{T}^{\ast}_{\Psi}M+N^{\ast}S_{\Delta}N>0 $.
	\end{itemize}
Then, $ S $ is the $ g- $frame operator of $ \{\Psi_{\zeta}M+\Delta_{\zeta}N\}_{\zeta\in \Theta} $.
\end{theorem}
\begin{proof}
	Let $ D_{1} $, $ D_{2} $ be $ g- $Bessel bounds of $ \Psi $ and $ \Delta $, respectively. Since $ M, $ $ N $ are bounded, for any $ x\in \mathcal{H} $, we have 
	\begin{align*}
		  \sum_{\zeta\in \Theta}  \langle( \Psi_{\zeta}M+ \Delta_{\zeta}N)x,( \Psi_{\zeta}M+ \Delta_{\zeta}N)x\rangle_{\mathcal{A}}&\leq  2\left(\sum_{\zeta\in \Theta}  \langle  \Delta_{\zeta}Nx,\Delta_{\zeta}N x\rangle_{\mathcal{A}}+ \sum_{\zeta\in \Theta}  \langle  \Psi_{\zeta}M x,\Psi_{\zeta}M x\rangle_{\mathcal{A}}\right)\\
		&\leq 2(D_{1} \Vert M \Vert^{2}+D_{2} \Vert N \Vert^{2})\langle x,x \rangle_{\mathcal{A}}.
	\end{align*} 
Then,  $ \{\Psi_{\zeta}M+\Delta_{\zeta}N\}_{\zeta\in \Theta} $ is a  $ g- $Bessel sequence for $ \mathcal{H} $. 

If $ \mathcal{T} $ is the synthesis operator of $ \{\Psi_{\zeta}M+\Delta_{\zeta}N\}_{\zeta\in \Theta} $, then we have 
\begin{equation*}
	\mathcal{T}^{\ast}x =  \{(\Psi_{\zeta}M+\Delta_{\zeta}N)x\}_{\zeta\in \Theta} =\mathcal{T}_{\Psi}^{\ast} M x+\mathcal{T}_{\Delta}^{\ast} N x,\;\forall x\in \mathcal{H}.
\end{equation*} 
i.e. $ \mathcal{T}= M^{\ast}\mathcal{T}_{\Psi} +N^{\ast}\mathcal{T}_{\Delta}   $.

 By Lemma \ref{L33}, the conditions $ (1) $ and $ (2) $ are equivalents. Moreover, we have 
 \begin{equation*}
 	S=\mathcal{T} \mathcal{T}^{\ast}= M^{\ast}S_{\Psi}M+M\mathcal{T}_{\Psi}\mathcal{T}^{\ast}_{\Delta}N+ N^{\ast}\mathcal{T}_{\Delta}\mathcal{T}^{\ast}_{\Psi}M+N^{\ast}S_{\Delta}N>0.
 \end{equation*}
Therefore, the conditions $ (2) $ and $ (3) $ are equivalent. 

\end{proof}
\begin{corollary} Let $ \Psi= \{\Psi_{\zeta}\} _{\zeta\in \Theta}$
	and $  \Delta= \{\Delta_{\zeta}\}_{\zeta\in \Theta} $ are two $ g- $Bessel sequences for $ \mathcal{H} $ with the synthesis operators $ \mathcal{T}_{\Psi} $, $ \mathcal{T}_{\Delta} $ and the $ g- $frame operators $ S_{\Psi} $, $ S_{\Delta} $, respectively. If $ \mathcal{T}_{\Psi}\mathcal{T}^{\ast}_{\Delta} $ is positive, then  $ \{\Psi_{\zeta}+\Delta_{\zeta}\}_{\zeta\in \Theta} $ is a $ g- $frame for $ \mathcal{H} $ with respect to $ \{	\mathcal{H}_{\zeta}\}_{\zeta\in \Theta} $.
\end{corollary}

\begin{proof}
	If $ M=N=I $ in Theorem \ref{t3}, then the $ g- $frame operator of the $ g- $Bessel sequence  $ \{\Psi_{\zeta}+\Delta_{\zeta}\}_{\zeta\in \Theta} $ is as follows:
	\begin{align*}
	    S&= \mathcal{T}_{\Psi}\mathcal{T}_{\Psi}^{\ast}+\mathcal{T}_{\Psi}\mathcal{T}^{\ast}_{\Delta}+\mathcal{T}_{\Delta}\mathcal{T}^{\ast}_{\Psi} \\
	   &= S_{\Psi}+\mathcal{T}_{\Psi}\mathcal{T}_{\Delta}^{\ast}+\mathcal{T}_{\Delta}\mathcal{T}_{\Psi}^{\ast}+S_{\Delta}>0.
	 \end{align*}  
\end{proof}
 Let $ \Psi= \{\Psi_{\zeta}\} _{\zeta\in \Theta}$
and $  \Delta= \{\Delta_{\zeta}\}_{\zeta\in \Theta} $ are two $ g- $Bessel sequences for $ \mathcal{H} $. 
In the following theorems, We give sufficient conditions for the sequences $ \{\theta_{\zeta}\}_{\zeta\in \Theta} $ and $ \{\delta_{\zeta}\}_{\zeta\in \Theta} $  which imply $ \{ \theta_{\zeta}\Psi_{\zeta}+\delta_{\zeta}\Delta_{\zeta}\}_{\zeta\in \Theta} $ is a $ g- $frame for $ \mathcal{H} $. 

\begin{theorem} \label{t7}Let $ \{\Psi_{\zeta} \in \operatorname{End}_{\mathcal{A}}^{\ast}(\mathcal{H},\mathcal{H}_{\zeta}): \zeta\in \Theta\} $ be a $ g- $frame for $ \mathcal{H} $ with bound $ D$ and $ D^{\prime}$, let $ \{\Delta_{\zeta}\}_{\zeta\in \Theta} $ be a $ g- $Bessel sequence for $ \mathcal{H} $ with Bessel bound $ D_{\Delta} $. Suppose that  $ \{\theta_{\zeta}\}_{\zeta\in \Theta} $ and $ \{\delta_{\zeta}\}_{\zeta\in \Theta} $
are two sequences from the algebra $\mathcal{A}$ such that $ 0< A<\vert \theta_{\zeta}\vert^{2} , \vert \delta_{\zeta}\vert^{2}<B< \infty$. If $ BD_{\Delta} < AD$, then $ \{ \theta_{\zeta}\Psi_{\zeta}+\delta_{\zeta}\Delta_{\zeta}\}_{\zeta\in \Theta} $ is a $ g- $frame for $ \mathcal{H} $ with respect to $\{ \mathcal{H}_{\zeta}\}_{\zeta\in \Theta} $. 
\end{theorem}
 \begin{proof}
 	We have for any $ x\in\mathcal{H} $, 
 	\begin{align*}
 		 \sum_{\zeta\in \Theta}  \langle( \theta_{\zeta} \Psi_{\zeta}+ \delta_{\zeta} \Delta_{\zeta})x,(\theta_{\zeta} \Psi_{\zeta}+ \delta_{\zeta}\Delta_{\zeta})x\rangle_{\mathcal{A}}
  &\leq  2\left(\sum_{\zeta\in \Theta}  \langle  \delta_{\zeta}\Delta_{\zeta}x,\delta_{\zeta}\Delta_{\zeta} x\rangle_{\mathcal{A}}+ \sum_{\zeta\in \Theta}  \langle \theta_{\zeta} \Psi_{\zeta}x,\theta_{\zeta} \Psi_{\zeta} x\rangle_{\mathcal{A}}\right)\\
 		&\leq2 B( D_{\Delta}+D^{\prime})\langle x,x\rangle_{\mathcal{A}}.
 	\end{align*} 
And for every $ x\in\mathcal{H} $, we obtain 
 \begin{align*}
 \Vert \sum_{\zeta\in \Theta}  \langle( \theta_{\zeta} \Psi_{\zeta}+ \delta_{\zeta} \Delta_{\zeta})x,(\theta_{\zeta} \Psi_{\zeta}+ \delta_{\zeta}\Delta_{\zeta})x\rangle_{\mathcal{A}}\Vert^{\frac{1}{2}} &\geq \Vert \sum_{\zeta\in \Theta}  \langle  (\theta_{\zeta}\Psi_{\zeta})x,(\theta_{\zeta}\Psi_{\zeta}) x\rangle_{\mathcal{A}}\Vert^{\frac{1}{2}} \\ &-\Vert \sum_{\zeta\in \Theta}  \langle(  \delta_{\zeta} \Delta_{\zeta})x,( \delta_{\zeta}\Delta_{\zeta})x\rangle_{\mathcal{A}}\Vert^{\frac{1}{2}}\\
 	  	&\geq (\sqrt{AD}-\sqrt{BD_{\Delta}})\Vert x \Vert.
 	  \end{align*} 
 	 
Therefore $ \{ \theta_{\zeta}\Psi_{\zeta}+\delta_{\zeta}\Delta_{\zeta}\}_{\zeta\in \Theta} $ is a $ g- $frame for $ \mathcal{H} $ with respect to $\{ \mathcal{H}_{\zeta}\}_{\zeta\in \Theta} $.
 	\end{proof}
 	\begin{theorem} \label{t11}Let $ \Psi= \{\Psi_{\zeta}\} _{\zeta\in \Theta}$
 		and $  \Delta= \{\Delta_{\zeta}\}_{\zeta\in \Theta} $ be two $ g- $frames for $ \mathcal{H} $ with the synthesis operators $ \mathcal{T}_{\Psi} $, $ \mathcal{T}_{\Delta} $, respectively. Suppose that  $ \{\theta_{\zeta}\}_{\zeta\in \Theta} $ and $ \{\delta_{\zeta}\}_{\zeta\in \Theta} $
 		are two sequences from the algebra $\mathcal{A}$ such that $ 0< A<\vert \theta_{\zeta}\vert^{2} , \vert \delta_{\zeta}\vert^{2}<B< \infty$. If $ \mathcal{T}_{\Psi}\mathcal{T}^{\ast}_{\Delta} $ is positive, then  $ \{\theta_{\zeta}\Psi_{\zeta}+\delta_{\zeta}\Delta_{\zeta}\}_{\zeta\in \Theta} $ is a $ g- $frame for $ \mathcal{H} $ with respect to $ \{	\mathcal{H}_{\zeta}\}_{\zeta\in \Theta} $.
 		
 	\end{theorem}
 \begin{proof}As in the proof of Theorem \ref{t7} $ \{ \theta_{\zeta}\Psi_{\zeta}+\delta_{\zeta}\Delta_{\zeta}\}_{\zeta\in \Theta} $ is a $ g- $Bessel sequence for $ \mathcal{H} $ with respect to $ \{	\mathcal{H}_{\zeta}\}_{\zeta\in \Theta} $. For every $ x\in\mathcal{H} $ 
 	\begin{align*}
 	 \sum_{\zeta\in \Theta}  \langle( \theta_{\zeta} \Psi_{\zeta}+ \delta_{\zeta} \Delta_{\zeta})x,(\theta_{\zeta} \Psi_{\zeta}+ \delta_{\zeta}\Delta_{\zeta})x\rangle_{\mathcal{A}}&=  \sum_{\zeta\in \Theta}  \langle  \delta_{\zeta}\Delta_{\zeta}x,\delta_{\zeta}\Delta_{\zeta} x\rangle_{\mathcal{A}}+ \sum_{\zeta\in \Theta}  \langle \theta_{\zeta} \Psi_{\zeta}x,\theta_{\zeta} \Psi_{\zeta} x\rangle_{\mathcal{A}}\\&+  \sum_{\zeta\in \Theta}  \langle  \theta_{\zeta}\Psi_{\zeta}x,\delta_{\zeta}\Delta_{\zeta} x\rangle_{\mathcal{A}}+
 	 \sum_{\zeta\in \Theta}  \langle  \delta_{\zeta}\Delta_{\zeta}x,\theta_{\zeta}\Psi_{\zeta} x\rangle_{\mathcal{A}}\\
 	& \geq A\left(  \sum_{\zeta\in \Theta}  \langle  \Delta_{\zeta}x,\Delta_{\zeta} x\rangle_{\mathcal{A}}+ \sum_{\zeta\in \Theta}  \langle  \Psi_{\zeta}x, \Psi_{\zeta} x\rangle_{\mathcal{A}}\right)\\
 	&+ \min_{\zeta \in \Theta} \lbrace\lvert \delta_{\zeta}\theta_{\zeta}\rvert,\lvert \theta_{\zeta}\delta_{\zeta}\rvert \rbrace\left( \sum_{\zeta\in \Theta}  \langle  \Psi_{\zeta}x,\Delta_{\zeta} x\rangle_{\mathcal{A}}+
 	 \sum_{\zeta\in \Theta}  \langle  \Delta_{\zeta}x,\Psi_{\zeta} x\rangle_{\mathcal{A}} \right)\\
 	&=A \left( \sum_{\zeta\in \Theta}  \langle  \Delta_{\zeta}x,\Delta_{\zeta} x\rangle_{\mathcal{A}}+ \sum_{\zeta\in \Theta}  \langle  \Psi_{\zeta}x, \Psi_{\zeta} x\rangle_{\mathcal{A}} \right)\\
 	&+\min_{\zeta \in \Theta}\lbrace\lvert \delta_{\zeta}\theta_{\zeta}\rvert,\lvert \theta_{\zeta}\delta_{\zeta}\rvert \rbrace \left(\langle \mathcal{T}_{\Psi}\mathcal{T}^{\ast}_{\Delta}x,x\rangle_{{\mathcal{A}}} +\langle \mathcal{T}_{\Delta}\mathcal{T}^{\ast}_{\Psi}x,x\rangle_{{\mathcal{A}}} \right).
 \end{align*} 
Then, we have 
\begin{equation*}
	 \sum_{\zeta\in \Theta}  \langle( \theta_{\zeta} \Psi_{\zeta}+ \delta_{\zeta} \Delta_{\zeta})x,(\theta_{\zeta} \Psi_{\zeta}+ \delta_{\zeta}\Delta_{\zeta})x\rangle_{\mathcal{A}}\geq A \left( \sum_{\zeta\in \Theta}  \langle  \Psi_{\zeta}x, \Psi_{\zeta} x\rangle_{\mathcal{A}}+\sum_{\zeta\in \Theta}  \langle  \Delta_{\zeta}x,\Delta_{\zeta} x\rangle_{\mathcal{A}}\right)\geq A(\alpha+\beta)\langle x,x\rangle_{\mathcal{A}}.
\end{equation*}	
Where $ \alpha $ and $\beta $  is a lower $ g- $frames bounds of $ \Psi $ and $\Delta$ respectively.
 \end{proof}
 	
\begin{corollary}Let $ \Psi= \{\Psi_{\zeta}\} _{\zeta\in \Theta}$
	and $  \Delta= \{\Delta_{\zeta}\}_{\zeta\in \Theta} $ are two tight $ g- $frame for $ \mathcal{H} $ with the synthesis operators $ \mathcal{T}_{\Psi} $, $ \mathcal{T}_{\Delta} $  respectively. If $ \mathcal{T}_{\Psi}\mathcal{T}^{\ast}_{\Delta}=0 $ , then  $ \{\Psi_{\zeta}+\Delta_{\zeta}\}_{\zeta\in \Theta} $ is a tight $ g- $frame for $ \mathcal{H} $ with respect to $ \{	\mathcal{H}_{\zeta}\}_{\zeta\in \Theta} $.
	
\end{corollary}
 	
 	\begin{proof} Using the proof of Theorem \ref{t11}, for every $ x\in\mathcal{H}  $ we have 
 		\begin{align*}
 		 \sum_{\zeta\in \Theta}  \langle( \Psi_{\zeta}+  \Delta_{\zeta})x,( \Psi_{\zeta}+ \Delta_{\zeta})x\rangle_{\mathcal{A}} &	= ( \sum_{\zeta\in \Theta}  \langle  \Delta_{\zeta}x,\Delta_{\zeta} x\rangle_{\mathcal{A}}+ \sum_{\zeta\in \Theta}  \langle  \Psi_{\zeta}x, \Psi_{\zeta} x\rangle_{\mathcal{A}})\\ 
 		&=( \alpha_{1}+\alpha_{2})\langle x,x\rangle_{\mathcal{A}}.
 		\end{align*}
 		Where $  \sum\limits_{\zeta\in \Theta}  \langle  \Psi_{\zeta}x,\Psi_{\zeta} x\rangle_{\mathcal{A}}= \alpha_{1}\langle x,x\rangle_{\mathcal{A}} $ and $ \sum\limits_{\zeta\in \Theta}  \langle  \Delta_{\zeta}x,\Delta_{\zeta} x\rangle_{\mathcal{A}}= \alpha_{2}\langle x,x\rangle_{\mathcal{A}}. $
 	\end{proof}
 	 \begin{theorem} Let $ \Psi= \{\Psi_{\zeta}\} _{\zeta\in \Theta}$ be a $ g- $frame for $ \mathcal{H} $ and $  \Delta= \{\Delta_{\zeta}\}_{\zeta\in \Theta} $  be a $ g- $Bessel sequence for $ \mathcal{H} $ with the synthesis operators $ \mathcal{T}_{\Psi} $, $ \mathcal{T}_{\Delta} $, respectively and  $ \mathcal{T}_{\Psi}\mathcal{T}^{\ast}_{\Delta} $ is positive. If $ \Lambda\in \operatorname{End}_{\mathcal{A}}^{\ast}(\mathcal{H})  $ is an isometry, then  $ \{\Psi_{\zeta}+\Delta_{\zeta}\}_{\zeta\in \Theta} $ is a $ g- $frame for $ \mathcal{H} $ with respect to $ \{	\mathcal{H}_{\zeta\in \Theta}\}_{\zeta\in \Theta} $.
 	 	
 	 	\end{theorem}
 
 \begin{proof} 
 	Similar to the proof of Theorem \ref{t3}, let $ A $ be a lower $ g- $frame bound of $ \Psi $ and $ \{(\Psi_{\zeta}+\Delta_{\zeta})\Lambda\}_{\zeta\in \Theta} $ is a $ g- $Bessel sequence for $ \mathcal{H} $ with respect to $ \{	\mathcal{H}_{\zeta}\}_{\zeta\in \Theta} $.
 For every $ x\in\mathcal{H} $, we have 
 	\begin{align*}
 		 \sum_{\zeta\in \Theta} \langle (\Psi_{\zeta}+\Delta_{\zeta})\Lambda,(\Psi_{\zeta}+\Delta_{\zeta})\Lambda \rangle_{{\mathcal{A}}}  &=   \sum_{\zeta\in \Theta}  \langle  \Delta_{\zeta}\Lambda x,\Delta_{\zeta} \Lambda x\rangle_{\mathcal{A}}+ \sum_{\zeta\in \Theta}  \langle \Psi_{\zeta}\Lambda x, \Psi_{\zeta} \Lambda x\rangle_{\mathcal{A}}\\&+  \sum_{\zeta\in \Theta}  \langle  \Psi_{\zeta}\Lambda x,\Delta_{\zeta}\Lambda x\rangle_{\mathcal{A}}+
 		 \sum_{\zeta\in \Theta}  \langle  \Delta_{\zeta} \Lambda x,\Psi_{\zeta} \Lambda x\rangle_{\mathcal{A}}\\
 		&=  \sum_{\zeta\in \Theta}  \langle  \Delta_{\zeta}\Lambda x,\Delta_{\zeta}\Lambda x\rangle_{\mathcal{A}}+ \sum_{\zeta\in \Theta}  \langle  \Psi_{\zeta}\Lambda x, \Psi_{\zeta} \Lambda x\rangle_{\mathcal{A}}\\
 		&+\langle \mathcal{T}_{\Psi}\mathcal{T}^{\ast}_{\Delta}\Lambda x,\Lambda x\rangle_{{\mathcal{A}}} +\langle \mathcal{T}_{\Delta}\mathcal{T}^{\ast}_{\Psi}\Lambda x,\Lambda x\rangle_{{\mathcal{A}}} .
 	\end{align*}
 So, 
 
 \begin{equation*}
 		 \sum_{\zeta\in \Theta} \langle (\Psi_{\zeta}+\Delta_{\zeta})\Lambda,(\Psi_{\zeta}+\Delta_{\zeta})\Lambda \rangle_{{\mathcal{A}}}  \geq  \sum_{\zeta\in \Theta}  \langle  \Psi_{\zeta}\Lambda x, \Psi_{\zeta} \Lambda x\rangle_{\mathcal{A}} \geq A \Vert \Lambda x \Vert^{2} =A\langle x,x\rangle_{\mathcal{A}},
 \end{equation*} 

 \end{proof}

\begin{theorem}
	
	Let $ \Psi=\{\Psi_{\zeta} \in \operatorname{End}_{\mathcal{A}}^{\ast}(\mathcal{H},\mathcal{H}_{\zeta}): \zeta\in \Theta\} $ be a $ g- $frame for $ \mathcal{H} $ with bounds $D$ and $D^{\prime}$, let $ \{\Delta_{\zeta}\}_{\zeta\in \Theta}$ be a $ g- $Bessel sequence for $ \mathcal{H} $ with Bessel bound $ D_{\Delta} $. If $ M,N\in\operatorname{End}_{\mathcal{A}}^{\ast}(\mathcal{H}) $ such that for any $ x \in \mathcal{H}$, $ \Vert Nx\Vert> \lambda \Vert x\Vert $ for some $ \lambda>0 $ and $ D_{\Delta}<D $, then $ \{ \Psi_{\zeta}M+\Delta_{\zeta}N\}_{\zeta\in \Theta} $ is a $ g- $ frame for  $ \mathcal{H} $ with respect to $\{ \mathcal{H}_{\zeta}\}_{\zeta\in \Theta} $.
	
\end{theorem}
\begin{proof}	Similar to the proof of Theorem \ref{t3}, $ \{(\Psi_{\zeta}+\Delta_{\zeta})M\}_{\zeta\in \Theta} $ is a $ g- $Bessel sequence for $ \mathcal{H} $ with respect to $ \{	\mathcal{H}_{\zeta}\}_{\zeta\in \Theta} $. For any $ x\in\mathcal{H} $, we obtain 
	\begin{align*}
	\Vert \sum_{\zeta\in \Theta} \langle (\Psi_{\zeta}+\Delta_{\zeta})M,(\Psi_{\zeta}+\Delta_{\zeta})M \rangle_{{\mathcal{A}}} \Vert^{\frac{1}{2}}  &\geq   \Vert\sum_{\zeta\in \Theta}  \langle \Psi_{\zeta}M x, \Psi_{\zeta} Mx\rangle_{\mathcal{A}}\Vert^{\frac{1}{2}}  - \Vert\sum_{\zeta\in \Theta}  \langle  \Delta_{\zeta}Nx,\Delta_{\zeta} N x\rangle_{\mathcal{A}}\Vert^{\frac{1}{2}}  \\&\geq  \sqrt{D}  \Vert   M x\Vert -\sqrt{D_{\Delta}} 
	\Vert N x\Vert \\
	&\geq (\sqrt{D}-\sqrt{D_{\Delta}})\Vert Nx \Vert  \\
 &\geq \lambda(\sqrt{D}-\sqrt{D_{\Delta}})\Vert x \Vert  .
\end{align*}
Therefore, $ \lambda^{2}(\sqrt{D}-\sqrt{D_{\Delta}})^{2} $ is a lower $ g- $frame bound for $ \{(\Psi_{\zeta}+\Delta_{\zeta})M\}_{\zeta\in \Theta} $. 
\end{proof}

\begin{corollary}	Let $ \{\Psi_{\zeta} \in \operatorname{End}_{\mathcal{A}}^{\ast}(\mathcal{H},\mathcal{H}_{\zeta}): \zeta\in \Theta\} $ be a $ g- $frame for $ \mathcal{H} $ with bound $ D $ and  $D^{\prime}$, let $ \{\Delta_{\zeta}\}_{\zeta\in \Theta}$ be a $ g- $Bessel sequence for $ \mathcal{H} $ with Bessel bound $ D_{\Delta} $. If $ D_{\Delta}<D $, then, $ \{\Psi_{\zeta}+\Delta_{\zeta}\}_{\zeta\in \Theta} $ is a $ g- $frame for $ \mathcal{H} $  
with respect to $\{ \mathcal{H}_{\zeta}\}_{\zeta\in \Theta} $.	
\end{corollary}
\begin{theorem}Let $\Psi= \{\Psi_{\zeta}\} _{\zeta\in \Theta}$ be a $ \alpha_{1}- $tight $ g- $frame for $ \mathcal{H} $ and $  \Delta= \{\Delta_{\zeta}\}_{\zeta\in \Theta} $ be a $ \alpha_{2}- $tight $ g- $frame for $ \mathcal{H} $ with the synthesis operators $ \mathcal{T}_{\Psi} $, $ \mathcal{T}_{\Delta} $, respectively, and $ \mathcal{T}_{\Psi}\mathcal{T}^{\ast}_{\Delta}=0 $. If $ M,N \in \operatorname{End}_{\mathcal{A}}^{\ast}(\mathcal{H}) $, then  $ \{\Psi_{\zeta}M+\Delta_{\zeta}N\}_{\zeta\in \Theta} $ is a $ \alpha- $tight $ g- $frame for $ \mathcal{H} $ with respect to $ \{	\mathcal{H}_{\zeta\in \Theta}\}_{\zeta\in \Theta} $ if and only if $ \alpha_{1}M^{\ast}M+\alpha_{2}N^{\ast}N= \alpha I $, for some $ \alpha>0 $.
	
\end{theorem}

\begin{proof}According to the proof of Theorem \ref{t3} $ \{(\Psi_{\zeta}+\Delta_{\zeta})M\}_{\zeta\in \Theta} $ is a $ g- $Bessel sequence for $ \mathcal{H} $. Since $ \mathcal{T}_{\Psi}\mathcal{T}^{\ast}_{\Delta}=0$, for any $ x\in\mathcal{H} $, we get $ \sum_{\zeta\in \Theta}  \Psi_{\zeta}^{\ast}\Delta_{\zeta}x =0$.

For every $ x\in \mathcal{H}$, we have
	\begin{equation*}
		\sum_{\zeta\in \Theta} \langle \Psi_{\zeta}M x, \Delta_{\zeta}Nx\rangle_{{\mathcal{A}}} = \sum_{\zeta\in \Theta} \langle M x,\Psi_{\zeta}^{\ast} \Delta_{\zeta}Nx\rangle_{{\mathcal{A}}} =\langle M x,\sum_{\zeta\in \Theta}  \Psi_{\zeta}^{\ast}\Delta_{\zeta}Nx\rangle_{{\mathcal{A}}}= 0.
	\end{equation*}
	
	We obtain, 
	\begin{align*}
	 \sum_{\zeta\in \Theta} \langle (\Psi_{\zeta}M+\Delta_{\zeta}N)x,(\Psi_{\zeta}M+\Delta_{\zeta}N)x \rangle_{{\mathcal{A}}}  &=   \sum_{\zeta\in \Theta}  \langle  \Delta_{\zeta}Nx,\Delta_{\zeta} N x\rangle_{\mathcal{A}}+ \sum_{\zeta\in \Theta}  \langle \Psi_{\zeta}M x, \Psi_{\zeta} M x\rangle_{{\mathcal{A}}}
\end{align*}
Then, we get
\begin{align*}
	   \sum_{\zeta\in \Theta}  \langle  \Delta_{\zeta}Nx,\Delta_{\zeta} N x\rangle_{\mathcal{A}}+ \sum_{\zeta\in \Theta}  \langle \Psi_{\zeta}M x, \Psi_{\zeta} M x\rangle_{{\mathcal{A}}}& = \alpha_{1}\Vert M x\Vert ^{2}+ \alpha_{2} \Vert Nx \Vert^{2} \\
	  &= \langle \alpha_{1}M^{*}M x,x\rangle_{{\mathcal{A}}}+ \langle \alpha_{2}N^{*}Nx,x\rangle_{{\mathcal{A}}}\\
	&= \langle \alpha_{1}M^{*}M x+\alpha_{2} N^{*}Nx,x\rangle_{{\mathcal{A}}}\\
		&= \langle \alpha Ix,x\rangle_{{\mathcal{A}}}\\
			&= \alpha  \langle x,x\rangle_{{\mathcal{A}}}.
\end{align*}
	\end{proof}
\begin{corollary}
Let $ \Psi= \{\Psi_{\zeta}\} _{\zeta\in \Theta}$ be a $ \alpha_{1}- $tight $ g- $frame for $ \mathcal{H} $ and $  \Delta= \{\Delta_{\zeta}\}_{\zeta\in \Theta} $ be a $ \alpha_{2}- $tight $ g- $frame for $ \mathcal{H} $ for $ \mathcal{H} $ with the synthesis operators $ \mathcal{T}_{\Psi} $, $ \mathcal{T}_{\Delta} $, respectively. If $ \mathcal{T}_{\Psi}\mathcal{T}^{\ast}_{\Delta}=0 $, then  $ \{\Psi_{\zeta}+\Delta_{\zeta}\}_{\zeta\in \Theta} $ is a $ \alpha- $tight $ g- $frame for $ \mathcal{H} $ with respect to $ \{	\mathcal{H}_{\zeta\in \Theta}\}_{\zeta\in \Theta} $ if and only if $ \alpha_{1}+\alpha_{2}= \alpha  $.
\end{corollary}

In what follows we study the stability of $ g- $frames for Hilbert $ C^{\ast}- $module $ \mathcal{H} $  under some of perturbations.
\begin{proposition} Let $ \Psi= \{\Psi_{\zeta}\} _{\zeta\in \Theta}$
 be a $ g- $frame for $ \mathcal{H} $ with bounds $ C $, $ D $ and $  \Delta= \{\Delta_{\zeta}\}_{\zeta\in \Theta} $. Let    $ \{\theta_{\zeta}\}_{\zeta\in \Theta} $ and $ \{\delta_{\zeta}\}_{\zeta\in \Theta} $
are two sequences from the algebra $\mathcal{A}$ such that $ 0< A<\vert \theta_{\zeta}\vert^{2} , \vert \delta_{\zeta}\vert^{2}<B< \infty$ and $ 0< \alpha_{1},\alpha_{2}<1 $ such that for any $ x\in \mathcal{H} $, we have   
\begin{align*}
\left( \Vert \sum_{\zeta\in \Theta}  \langle \theta_{\zeta} \Psi_{\zeta}x,\theta_{\zeta}\Psi_{\zeta} x\rangle_{\mathcal{A}}\Vert - \Vert\sum_{\zeta\in \Theta}  \langle \delta_{\zeta} \Delta_{\zeta}x,\delta_{\zeta}\Delta_{\zeta} x\rangle_{\mathcal{A}}\Vert\right)^{\frac{1}{2}}&\leq \alpha_{1}    \left( \Vert \sum_{\zeta\in \Theta}  \langle \theta_{\zeta} \Psi_{\zeta}x,\theta_{\zeta}\Psi_{\zeta} x\rangle_{\mathcal{A}}\Vert\right)^{\frac{1}{2}}\\	&+\alpha_{2}\left(\Vert\sum_{\zeta\in \Theta}  \langle \delta_{\zeta} \Delta_{\zeta}x,\delta_{\zeta}\Delta_{\zeta} x\rangle_{\mathcal{A}}\Vert\right)^{\frac{1}{2}}.
\end{align*}

 Then  $\{ \Delta_{\zeta}\}_{\zeta\in \Theta} $ is a $ g- $frame for $ \mathcal{H} $.

\end{proposition}

\begin{proof} For any $ x\in\mathcal{H} $, 
	\begin{align*}
		\left(\Vert \sum_{\zeta\in \Theta}  \langle \theta_{\zeta} \Psi_{\zeta}x,\theta_{\zeta}\Psi_{\zeta} x\rangle_{\mathcal{A}}\Vert\right)^{\frac{1}{2}} &\leq \left( \Vert \sum_{\zeta\in \Theta}  \langle \theta_{\zeta} \Psi_{\zeta}x,\theta_{\zeta}\Psi_{\zeta} x\rangle_{\mathcal{A}} \Vert - \Vert \sum_{\zeta\in \Theta}  \langle \delta_{\zeta} \Delta_{\zeta}x,\delta_{\zeta}\Delta_{\zeta} x\rangle_{\mathcal{A}}\Vert \right)^{\frac{1}{2}}\\
		&+\left(\Vert\sum_{\zeta\in \Theta}  \langle \delta_{\zeta} \Delta_{\zeta}x,\delta_{\zeta}\Delta_{\zeta} x\rangle_{\mathcal{A}}\Vert\right)^{\frac{1}{2}}\\
		&\leq \alpha_{1} \left( \Vert\sum_{\zeta\in \Theta}  \langle \theta_{\zeta} \Psi_{\zeta}x,\theta_{\zeta}\Psi_{\zeta} x\rangle_{\mathcal{A}}\Vert\right)^{\frac{1}{2}}	+\alpha_{2}\left(\Vert\sum_{\zeta\in \Theta}  \langle \delta_{\zeta} \Delta_{\zeta}x,\delta_{\zeta}\Delta_{\zeta} x\rangle_{\mathcal{A}}\Vert\right)^{\frac{1}{2}}\\
		&+\left(\Vert\sum_{\zeta\in \Theta}  \langle \delta_{\zeta} \Delta_{\zeta}x,\delta_{\zeta}\Delta_{\zeta} x\rangle_{\mathcal{A}}\Vert\right)^{\frac{1}{2}}.
	\end{align*}
	Then, we obtain 
	\begin{equation*}
		\left(1-\alpha_{1}\right)	\left( \sum_{\zeta\in \Theta}  \langle \theta_{\zeta} \Psi_{\zeta}x,\theta_{\zeta}\Psi_{\zeta} x\rangle_{\mathcal{A}}\right)^{\frac{1}{2}}\leq \left(1+\alpha_{2}\right)\left(\sum_{\zeta\in \Theta}  \langle \delta_{\zeta} \Delta_{\zeta}x,\delta_{\zeta}\Delta_{\zeta} x\rangle_{\mathcal{A}}\right)^{\frac{1}{2}}.
	\end{equation*}
So, we have
\begin{equation*}
	CA\left(1-\alpha_{1}\right)^{2} \langle x, x\rangle_{\mathcal{A}}\leq B \left(1+\alpha_{2}\right)^{2}\left(\sum_{\zeta\in \Theta}  \langle  \Delta_{\zeta}x,\Delta_{\zeta} x\rangle_{\mathcal{A}}\right).
\end{equation*} 
This implies that $  \sum_{\zeta\in \Theta}  \langle  \Delta_{\zeta}x,\Delta_{\zeta} x\rangle_{\mathcal{A}} \geq \dfrac{CA\left(1-\alpha_{1}\right)^{2}}{ B \left(1+\alpha_{2}\right)^{2}}\langle x, x\rangle_{\mathcal{A}} $, $ \forall x\in\mathcal{H} $.

On the other hand, for any $ x\in\mathcal{H} $ we get 
	\begin{equation*}
	\left(1-\alpha_{2}\right)	\left( \sum_{\zeta\in \Theta}  \langle \delta_{\zeta}\Delta_{\zeta}x,\delta_{\zeta}\Delta_{\zeta} x\rangle_{\mathcal{A}}\right)^{\frac{1}{2}}\leq \left(1+\alpha_{1}\right)\left(\sum_{\zeta\in \Theta}  \langle \theta_{\zeta} \Psi_{\zeta}x,\ \theta_{\zeta} \Psi_{\zeta}x\rangle_{\mathcal{A}}\right)^{\frac{1}{2}}.
\end{equation*}
So, we obtain 
\begin{equation*}
	A\left(1-\alpha_{2}\right)^{2}\left(\sum_{\zeta\in \Theta}  \langle  \Delta_{\zeta}x,\Delta_{\zeta} x\rangle_{\mathcal{A}}\right)\leq \langle x, x\rangle_{\mathcal{A}} DB \left(1+\alpha_{2}\right)^{2}.
\end{equation*}
Therefore,   $ \sum_{\zeta\in \Theta}  \langle  \Delta_{\zeta}x,\Delta_{\zeta} x\rangle_{\mathcal{A}}\leq \dfrac{BD \left(1+\alpha_{2}\right)^{2}}{A\left(1-\alpha_{1}\right)^{2} } \langle x, x\rangle_{\mathcal{A}}$, $ \forall x\in\mathcal{H} $.

\end{proof}

\begin{theorem}Let $ \Psi= \{\Psi_{\zeta}\} _{\zeta\in \Theta}$
	be a $ g- $frame for $ \mathcal{H} $ with bounds $ C $, $ D $ and $  \Delta= \{\Delta_{\zeta}\}_{\zeta\in \Theta} $. Let    $ \{\theta_{\zeta}\}_{\zeta\in \Theta} $ and $ \{\delta_{\zeta}\}_{\zeta\in \Theta} $ are two sequences from the algebra $\mathcal{A}$ such that $ 0< A< \theta_{\zeta}^{2} ,  \delta_{\zeta}^{2}<B< \infty$ and $ 0< \alpha_{1},\alpha_{2}<1 $ such that for any $ x\in \mathcal{H} $,
	\begin{equation*}
\sum_{\zeta\in \Theta}  \langle (\theta_{\zeta} \Psi_{\zeta}-\delta_{\zeta} \Delta_{\zeta})x,(\theta_{\zeta}\Psi_{\zeta}-\delta_{\zeta} \Delta_{\zeta}) x\rangle_{\mathcal{A}}\leq \alpha_{1} (  \sum_{\zeta\in \Theta}  \langle \theta_{\zeta} \Psi_{\zeta}x,\theta_{\zeta}\Psi_{\zeta} x\rangle_{\mathcal{A}})	+\alpha_{2}( \sum_{\zeta\in \Theta}  \langle \delta_{\zeta} \Delta_{\zeta}x,\delta_{\zeta}\Delta_{\zeta} x\rangle_{\mathcal{A}})
	\end{equation*}
	
	then  $\{ \Delta_{\zeta}\}_{\zeta\in \Theta} $ is a $ g- $frame for $ \mathcal{H} $ with respect to $ \{	\mathcal{H}_{\zeta\in \Theta}\}_{\zeta\in \Theta} $.
	
	\end{theorem}
\begin{proof}For $x \in \mathcal{H} $, we have 
	\begin{align*}
		( \sum_{\zeta\in \Theta}  \langle \theta_{\zeta} \Psi_{\zeta}x,\theta_{\zeta}\Psi_{\zeta} x\rangle_{\mathcal{A}})^{\frac{1}{2}}&\leq	\Vert  \sum_{\zeta\in \Theta}  \langle (\theta_{\zeta} \Psi_{\zeta}-\delta_{\zeta} \Delta_{\zeta})x,(\theta_{\zeta}\Psi_{\zeta}-\delta_{\zeta} \Delta_{\zeta}) x\rangle_{\mathcal{A}}\Vert^{\frac{1}{2}} + \Vert\sum_{\zeta\in \Theta}  \langle \delta_{\zeta} \Delta_{\zeta}x,\delta_{\zeta}\Delta_{\zeta} x\rangle_{\mathcal{A}}\Vert^{\frac{1}{2}}\\
		&\leq( \alpha_{1} \Vert \sum_{\zeta\in \Theta}  \langle \theta_{\zeta} \Psi_{\zeta}x,\theta_{\zeta}\Psi_{\zeta} x\rangle_{\mathcal{A}}\Vert	+\alpha_{2}\Vert\sum_{\zeta\in \Theta}  \langle \delta_{\zeta} \Delta_{\zeta}x,\delta_{\zeta}\Delta_{\zeta} x\rangle_{\mathcal{A}}\Vert)^{\frac{1}{2}}\\
		&+( \Vert\sum_{\zeta\in \Theta}  \langle \delta_{\zeta} \Delta_{\zeta}x,\delta_{\zeta}\Delta_{\zeta} x\rangle_{\mathcal{A}}\Vert)^{\frac{1}{2}}\\
		&\leq (\alpha_{1}\Vert \sum_{\zeta\in \Theta}  \langle \theta_{\zeta} \Psi_{\zeta}x,\theta_{\zeta}\Psi_{\zeta} x\rangle_{\mathcal{A}}\Vert)^{\frac{1}{2}}+ 2(\alpha_{2}\Vert\sum_{\zeta\in \Theta}  \langle \delta_{\zeta} \Delta_{\zeta}x,\delta_{\zeta}\Delta_{\zeta} x\rangle_{\mathcal{A}}\Vert)^{\frac{1}{2}}.
	\end{align*} 
	Then, for any $ x\in\mathcal{H} $ we get
	\begin{equation*}
		(1-\sqrt{\alpha_{1}})(	\Vert \sum_{\zeta\in \Theta}  \langle \theta_{\zeta} \Psi_{\zeta}x,\theta_{\zeta}\Psi_{\zeta} x\rangle_{\mathcal{A}}\Vert)^{\frac{1}{2}}\leq (1-2\sqrt{\alpha_{2}})\Vert\sum_{\zeta\in \Theta}  \langle \delta_{\zeta} \Delta_{\zeta}x,\delta_{\zeta}\Delta_{\zeta} x\rangle_{\mathcal{A}}\Vert)^{\frac{1}{2}}.
	\end{equation*}
We obtain,

	  $\Vert \sum_{\zeta\in \Theta}  \langle  \Delta_{\zeta}x,\Delta_{\zeta} x\rangle_{\mathcal{A}}\Vert\geq \dfrac{CA (1+2\sqrt{\alpha_{2}})^{2}}{B(1-\sqrt{\alpha_{1}})^{2} } \langle x, x\rangle_{\mathcal{A}}$, $ \forall x\in\mathcal{H} $.
	
Similarly, for any $ x\in\mathcal{H} $, 

  $\Vert \sum_{\zeta\in \Theta}  \langle  \Delta_{\zeta}x,\Delta_{\zeta} x\rangle_{\mathcal{A}}\Vert\geq \dfrac{DB (1+2\sqrt{\alpha_{2}})^{2}}{C(1-\sqrt{\alpha_{1}})^{2} } \langle x, x\rangle_{\mathcal{A}}$, $ \forall x\in\mathcal{H} $.
\end{proof}

\begin{theorem} \label{t12}Let $ \Psi= \{\Psi_{\zeta}\} _{\zeta\in \Theta}$
be a $ g- $frame for $ \mathcal{H} $ with bounds $ C $, $ D $ $ (D>1) $ and assume that $  \Delta= \{\Delta_{\zeta}\}_{\zeta\in \Theta} $ is a family so that for any $ J\subset \Theta $ with $ \vert J\vert< \infty  $, we have 
\begin{equation*}
	\Vert \sum_{\zeta\in J}( \Psi_{\zeta}^{*}\Psi_{\zeta}-\Delta_{\zeta})x \Vert \leq \dfrac{C}{D}\Vert x\Vert, \;\forall x\in\mathcal{H}. 
\end{equation*}
	Then  $\{ \Delta_{\zeta}\}_{\zeta\in \Theta} $ is a $ g- $frame for $ \mathcal{H} $ with bounds $ \dfrac{1}{C}(\dfrac{D+1}{D} )$, $ \dfrac{C+D^{2}}{D} $.
\end{theorem}

\begin{proof}
	
	For any $ x\in\mathcal{H} $ we have
	
	\begin{align*}
		\Vert \sum_{\zeta\in J} \Psi_{\zeta}^{*}\Psi_{\zeta}x\Vert& = \sup_{y\in\mathcal{H}, \Vert y\Vert =1}\Vert \langle \sum_{\zeta\in J} \Psi_{\zeta}^{*}\Psi_{\zeta}x, y\rangle_{\mathcal{A}} \Vert\\
		&= \sup_{y\in\mathcal{H}, \Vert y\Vert =1}\Vert  \sum_{\zeta\in J}\langle \Psi_{\zeta}x, \Psi_{\zeta}y\rangle_{\mathcal{A}} \Vert\\
		&\leq \sup_{y\in\mathcal{H}, \Vert y\Vert =1}\Vert  \sum_{\zeta\in J}\langle \Psi_{\zeta}x, \Psi_{\zeta}x \rangle_{\mathcal{A}} \Vert^{\frac{1}{2}} \Vert  \sum_{\zeta\in J}\langle \Psi_{\zeta}y, \Psi_{\zeta}y\rangle_{\mathcal{A}} \Vert^{\frac{1}{2}}\\
		&\leq \sqrt{D} \Vert  \sum_{\zeta\in J}\langle \Psi_{\zeta}x, \Psi_{\zeta}x\rangle_{\mathcal{A}} \Vert^{\frac{1}{2}}\\
		&\leq D \Vert x\Vert .
	\end{align*}
So we obtain 

\begin{equation*}
		\Vert \sum_{\zeta\in J} \Delta_{\zeta}^{*}\Delta_{\zeta}x\Vert\leq (\dfrac{C}{D}+D) \Vert x\Vert. 
\end{equation*}
Hence, the series $ \sum_{\zeta\in \Theta}\Delta_{\zeta}^{*}\Delta_{\zeta}x $ converge. We define 
\begin{equation*}
	\mathcal{K}(x)= \sum_{\zeta\in \Theta} \Delta_{\zeta}^{*}\Delta_{\zeta}x ,\;\forall x\in\mathcal{H}.
	\end{equation*}
This operator is well defined and bounded on $ \mathcal{H} $ and $ \Vert \mathcal{K}  \Vert\leq   \dfrac{C}{D}+D$. 

Also for $ x\in\mathcal{H} $
\begin{align*}
		\Vert \sum_{\zeta\in \Theta} \Delta_{\zeta}x\Vert^{2} &= \Vert \sum_{\zeta\in \Theta}\langle \Delta_{\zeta}x,\Delta_{\zeta}x\rangle_{\mathcal{A}}\Vert\\
		&\leq \Vert \mathcal{K} \Vert \Vert x\Vert ^{2}.
\end{align*}
That is $ \{\Delta_{\zeta}\}_{\zeta\in \Theta} $ is a $ g- $Bessel sequence for $ \mathcal{H} $. If $ S $ is the $ g- $frame operator of $ \{\Psi_{\zeta}\}_{\zeta\in \Theta} $, we get for all $ x\in\mathcal{H} $ $$ \Vert Sx-\mathcal{K} x\Vert \leq \dfrac{C}{D} \Vert x\Vert.$$ 
Then, $ \Vert x- \mathcal{K} S^{-1}x \Vert \leq \dfrac{1}{D}\Vert x\Vert \Rightarrow  \Vert I- \mathcal{K} S^{-1} \Vert \leq \dfrac{1}{D}\leq 1$.
Hence, $ \mathcal{K} S^{-1} $ is invertible and so is $ \mathcal{K}  $. Therefore, $ \{\Delta_{\zeta}\}_{\zeta} $ is a $ g- $frame for $ \mathcal{H} $. Also, we have for $ x\in\mathcal{H} $
 \begin{equation*}
 \Vert	\mathcal{K} ^{-1}\Vert \leq \Vert S^{-1}\Vert \Vert S\mathcal{K} ^{-1}\Vert \leq \dfrac{1}{C}(\dfrac{D+1}{D}).
 \end{equation*}
Therefore, $ \dfrac{1}{C}(\dfrac{D+1}{D}) $, $ \dfrac{C}{A}+D $ are bounds for $ \{\Delta_{\zeta}\}_{\zeta} $.
	\end{proof}

\begin{corollary}Let $ \Psi= \{\Psi_{\zeta}\} _{\zeta\in \Theta}$
	be a $ g- $frame for $ \mathcal{H} $ with bounds $ C $, $ D $ $ (0<\alpha<D) $. Suppose $ \Delta= \{\Delta_{\zeta}\}_{\zeta\in \Theta} $ such that for $ x\in\mathcal{H}, $ 
$$ \Vert \sum_{\zeta\in \Theta}\left( \Psi_{\zeta}^{*}\Psi_{\zeta}-\Delta_{\zeta}^{*}\Delta_{\zeta}\right) x  \Vert\leq \alpha \Vert x \Vert. $$ Then  $\{ \Delta_{\zeta}\}_{\zeta\in \Theta} $ is a $ g- $frame for $ \mathcal{H} $ with respect to $ \{	\mathcal{H}_{\zeta}\}_{\zeta\in \Theta} $.
	\end{corollary}

\begin{proof}
	Similarly the proof of the  Theorem \ref{t12}, if we take
	
	  \begin{equation*}
	  	\mathcal{K} (x)= \sum_{\zeta\in \Theta} \Delta_{\zeta}^{*}\Delta_{\zeta}x ,\;\forall x\in\mathcal{H}.
	  \end{equation*}
  Then  $\{ \Delta_{\zeta}\}_{\zeta\in \Theta} $ is a $ g- $Bessel sequence for $ \mathcal{H} $, we have $$ \Vert Sx-\mathcal{K} x\Vert \leq \alpha \Vert x\Vert  .$$ Then 
  
  $$ \Vert x-\mathcal{K} S^{-1} x \Vert \leq \alpha \Vert S^{-1}\Vert  \Vert x\Vert \leq \alpha \dfrac{1}{C} \Vert x \Vert, \; \forall x\in\mathcal{H}.$$ 
  This implies that $$ \Vert I-\mathcal{K} S^{-1}\Vert <1. $$ So, $ \mathcal{K} S^{-1} $ is invertible and $ \mathcal{K}  $ is also invertible.
  
   Hence, $ g- $Bessel sequence $ \{\Delta_{\zeta}\}_{\zeta} $ is a $ g- $frame for $ \mathcal{H} $ with respect to $ \{\mathcal{H}\}_{\zeta\in \Theta} $.
  
\end{proof}

	\section{Examples in \texorpdfstring{$C^*$}{C*}-Algebraic Settings}\label{secExamples}

We now present two concrete examples showing how to realize g-frames and their sums in explicit $C^*$-algebraic modules. These examples illustrate that the sum of two g-frames can fail to remain a g-frame unless certain positivity or invertibility conditions are enforced.

\subsection{Example with matrix algebras}

\begin{example}\label{exMatrixAlgebra}
	Let $\mathcal{A} = M_2(\mathbb{C})$ be the $2\times 2$ complex matrices, which is a unital $C^*$-algebra with the usual operator norm and $^*$-operation. Consider the Hilbert $\mathcal{A}$-module $\mathcal{H} = \mathcal{A}^2$ (i.e., column 2-vectors over $\mathcal{A}$) with the standard $\mathcal{A}$-valued inner product
	\[
	\langle (a_1,a_2)^T,\;(b_1,b_2)^T\rangle_{\mathcal{A}}
	\;=\;
	a_1^* b_1 + a_2^* b_2,
	\]
	where addition and multiplication are in $M_2(\mathbb{C})$.
	
	Define two families $\{\Lambda_1,\Lambda_2\}$ and $\{\Theta_1,\Theta_2\}$ as follows. For $x\in\mathcal{H}$ (which is a pair $(x_1,x_2)^T$ with each $x_i\in M_2(\mathbb{C})$):
	\[
	\Lambda_1(x) = \bigl(x_1,\,0\bigr)^T, 
	\quad
	\Lambda_2(x) = \bigl(0,\,x_2\bigr)^T,
	\]
	\[
	\Theta_1(x) = \bigl(x_2,\,0\bigr)^T,
	\quad
	\Theta_2(x) = \bigl(0,\,x_1\bigr)^T.
	\]
	One checks that each $\{\Lambda_1,\Lambda_2\}$ and $\{\Theta_1,\Theta_2\}$ forms a g-frame for $\mathcal{H}$ (they essentially “capture” each component of $(x_1,x_2)^T$ in disjoint slots).
	
	Now consider the sum $\{\Lambda_1+\Theta_1,\;\Lambda_2+\Theta_2\}$. One quickly sees that
	\[
	(\Lambda_1+\Theta_1)(x) 
	\;=\;
	\bigl(x_1+x_2,\;0\bigr)^T,
	\quad
	(\Lambda_2+\Theta_2)(x)
	\;=\;\bigl(0,\;x_2+x_1\bigr)^T.
	\]
	\emph{Case 1:} If $x_1$ and $x_2$ are such that $x_1 = -x_2$, then $(\Lambda_1+\Theta_1)(x)=0$ and $(\Lambda_2+\Theta_2)(x)=0$, so the sum has no lower bound and is \emph{not} a g-frame in that arrangement.
	
	\emph{Case 2:} If we modify one family by a suitable invertible multiplier $L\in \mathrm{End}_\mathcal{A}^*(\mathcal{H})$ (for instance, letting $L$ act diagonally on the $\mathcal{A}^2$ coordinate), then positivity or invertibility conditions can ensure no such complete cancellation occurs, restoring a uniform lower bound.
	
	Hence, this example demonstrates how the sums of two g-frames \emph{may fail} to be a g-frame unless additional constraints (like invertibility or positivity) are imposed, exactly as our theorems suggest.
\end{example}

\subsection{Example with a continuous-function algebra}

\begin{example}\label{exContinuousFunctions}
	Let $\mathcal{A} = C([0,1])$ be the $C^*$-algebra of continuous complex-valued functions on $[0,1]$. We consider the standard \emph{free module} $\mathcal{H} = \mathcal{A}^m$ for some $m\in \mathbb{N}$. An element $x\in \mathcal{H}$ can be viewed as a vector $(f_1,\ldots,f_m)^T$ with each $f_j\in C([0,1])$. The $\mathcal{A}$-valued inner product is
	\[
	\langle x,y\rangle_{\mathcal{A}}(\alpha)
	\;=\;
	\sum_{j=1}^m \overline{f_j(\alpha)}\, g_j(\alpha),
	\]
	for $x=(f_1,\dots,f_m)$, $y=(g_1,\dots,g_m)$ in $\mathcal{H}$, where $(\langle x,y\rangle_{\mathcal{A}})(\alpha)$ is a continuous function of $\alpha\in [0,1]$.
	
	Define, for instance, $\Lambda_1(x) = (f_1, 0,\dots,0)$, $\Lambda_2(x) = (0,f_2,0,\dots,0)$, etc.\ so that $\{\Lambda_i\}_{i=1}^m$ collectively “read off” the coordinates. This forms a g-frame (essentially a standard basis in $\mathcal{A}^m$). Next, define $\Theta_1,\ldots,\Theta_m$ similarly but perhaps with a shift: $\Theta_j(x)=(0,\dots,0,f_j,0,\dots)$. Each family alone is a g-frame. However, if we sum them coordinate-wise, $\Lambda_j + \Theta_j$, there may be points $\alpha\in[0,1]$ at which $f_j(\alpha)$ and $g_j(\alpha)$ combine destructively (e.g.\ $f_j(\alpha) + g_j(\alpha)=0$), losing the lower bound in the entire module.
	
	On the other hand, if we multiply one family by an invertible element in $\mathrm{End}_\mathcal{A}^*(\mathcal{H})$ (for example, a bounded diagonal operator with no vanishing entries over $[0,1]$), then we can often ensure positivity of the new g-frame operator. This aligns with Theorem \ref{t3}, requiring that a certain operator be invertible in the Hilbert $C^*$-module sense.
\end{example}

These examples show concretely how sums of g-frames in a $C^*$-algebraic environment can fail or succeed depending on invertibility considerations.

\section{Conclusion and Remarks}

We have established several conditions under which the sum of two g-frames (or a g-frame and a g-Bessel sequence) in a Hilbert $C^*$-module remains a g-frame. Central to these conditions are the invertibility or positivity of the associated operators (for instance, $(I+L)^* S_\Lambda (I+L)$ or $M^* T_\Lambda + N^* T_\Theta$). Our explicit examples in matrix algebras and continuous function algebras illustrate how cancellation effects can destroy the uniform lower bound unless we impose these algebraic constraints. 

These results extend the known stability properties of g-frames and highlight the unifying role of invertible operators in controlling sums of g-frames in Hilbert $C^*$-modules.

\bibliographystyle{amsplain}

\end{document}